\documentclass[12pt,dvipdfmx]{amsart}

\usepackage{amssymb,hyperref}
\numberwithin{equation}{section}

\newtheorem{theorem}{Theorem}[section]

\newtheorem{lemma}[theorem]{Lemma}
\newtheorem{corollary}[theorem]{Corollary}

\theoremstyle{definition}
\newtheorem{definition}[theorem]{Definition}

\newtheorem{remark}[theorem]{Remark}

\newcommand{\ZZ}{ \ensuremath{\mathbb{Z}}}

\newcommand{\init}{\ensuremath{\mathrm{in}}\hspace{1pt}}

\newcommand{\rank}{\ensuremath{\mathrm{rank}}\hspace{1pt}}

\def\cocoa{{\hbox{\rm C\kern-.13em o\kern-.07em C\kern-.13em o\kern-.15em A}}}

\newcommand{\A}{\mathcal{A}}
\newcommand{\B}{\mathcal{B}}

\newcommand{\ee}{\mathbf{e}}
\newcommand{\width}{\mathrm{width}}
\newcommand{\hvec}{\mathbf{h}}

\begin{document}

\title{$h$-vectors of simplicial cell balls}

\author{Satoshi Murai}
\address{
Satoshi Murai,
Department of Mathematical Science,
Faculty of Science,
Yamaguchi University,
1677-1 Yoshida, Yamaguchi 753-8512, Japan.
}

%\thanks{This work was supported by KAKENHI 22740018}

%\author{Another Author}
%\address{
%}
%\email{}

%Keyword and Subject Classes (if needed)
%\keywords{}
%\subjclass[2000]{}

\begin{abstract}
A simplicial cell ball is a simplicial poset whose geometric realization is homeomorphic to a ball.
Recently, Samuel Kolins gave a series of necessary conditions and sufficient conditions on $h$-vectors of simplicial cell balls,
and characterized them up to dimension $6$.
In this paper,
we extend Kolins' results.
We characterize all possible $h$-vectors of simplicial cell balls
in arbitrary dimension.
\end{abstract}

\maketitle

\section{Introduction}

A \textit{simplicial poset} is a finite poset $P$ with a minimal element $\hat 0$
such that every interval $[\hat 0, \sigma]$ for $\sigma \in P$ is a Boolean algebra.
A typical example of a simplicial poset is the face poset of a simplicial complex,
but not all simplicial posets come from simplicial complexes.
Simplicial posets are CW-posets.
Thus for any simplicial poset $P$ there is a regular CW-complex $\Gamma(P)$ whose face poset is isomorphic to $P$ (see \cite[pp.\ 8--9]{Bj}).
%A regular CW-complex whose face poset is a simplicial poset
%is called a \textit{simplicial cell complex}.
A simplicial poset $P$ (or a CW-complex $\Gamma(P)$)
is called
a \textit{simplicial cell $d$-ball (respectively $d$-sphere)}
if $\Gamma(P)$ is homeomorphic to a $d$-ball (respectively $d$-sphere).
%is a simplicial poset $P$ such that $\Gamma(P)$ is homeomorphic to a $d$-ball (respectively $d$-sphere).
%A \textit{simplicial cell $d$-ball} is a simplicial poset $P$ such $\Gamma(P)$ is homeomorphic to a $d$-ball.
Kolins \cite{Ko} studied $h$-vectors of simplicial cell balls
and gave a characterization of $h$-vectors of simplicial cell balls up to dimension $6$.
In this paper,
we extend the results of Kolins
and give a complete characterization of $h$-vectors of simplicial cell balls.

Let $P$ be a simplicial poset.
We say that an element $\sigma \in P$ has \textit{rank $i$}, denoted $\rank \sigma =i$,
if $[\hat 0,\sigma]$ is a Boolean algebra of rank $i$.
The \textit{dimension} of $P$ is 
$$\dim P= \max \{\rank \sigma : \sigma \in P\}-1.$$
Let $d=\dim P +1$ and let $f_i=f_i(P)$ be the number of elements $\sigma \in P$ having rank $i+1$ for $i=-1,0,\dots,d-1$.
Thus $f_i$ is the number of $i$-dimensional cells of $\Gamma (P)$.
The vector $f(P)=(f_{-1},f_0,\dots,f_{d-1})$ is called the \textit{$f$-vector (face vector) of $P$}.
We define the \textit{$h$-vector $h(P)=(h_0(P),h_1(P),\dots,h_d(P)) \in \ZZ^{d+1}$ of $P$} by the relation
$$\sum_{i=0}^d f_{i-1}  (1-t)^{d-i} = \sum_{i=0}^d h_i(P) t^{d-i}.$$
Then knowing $f(P)$ is equivalent to knowing $h(P)$.

On face vectors of simplicial cell spheres,
Stanley \cite{St} and Masuda \cite{Ma}
proved the following result,
which characterizes all possible $h$-vectors of simplicial cell spheres.

\begin{theorem}[Stanley, Masuda]
\label{1.1}
Let $\hvec =(h_0,h_1,\dots,h_d) \in \ZZ^{d+1}$.
Then $\hvec$ is the $h$-vector of a simplicial cell $(d-1)$-sphere if and only if
it satisfies the following conditions:
\begin{itemize}
\item[(1)] $h_0=h_d=1$ and $h_i=h_{d-i}$ for $i=1,2,\dots,d-1$.
\item[(2)] $h_i \geq 0$ for $i=0,1,\dots,d$.
\item[(3)] If $h_n=0$ for some $1 \leq n \leq d-1$ then $\sum_{k=0}^d h_k$ is even.
\end{itemize} 
\end{theorem}

For a vector $\hvec=(h_0,h_1,\dots,h_d) \in \ZZ^{d+1}$,
we define $\partial \hvec=(\partial h_0,\partial h_1,\dots,\partial h_{d-1}) \in \ZZ^d$ by
$$\partial h_i = (h_0+h_1+\cdots +h_i)-(h_d+h_{d-1}+\cdots +h_{d-i})$$
for $i=0,1,\dots,d-1$.
It is known that if $\hvec$ is the $h$-vector of a simplicial cell $(d-1)$-ball $P$,
then $\partial \hvec$ is the $h$-vector of the boundary sphere of $P$.
In this paper,
we prove the next result,
which characterizes all possible $h$-vectors of simplicial cell balls.

\begin{theorem}
\label{1.2}
Let $\hvec=(h_0,h_1,\dots,h_d) \in \ZZ^{d+1}$.
Then $\hvec$ is the $h$-vector of a simplicial cell $(d-1)$-ball if and only if
it satisfies the following conditions:
\begin{itemize}
\item[(1)] $h_0=1,$ $h_d=0$ and $h_k \geq 0$ for $k=1,2,\dots,d-1$.
\item[(2)] $\partial h_k \geq 0$ for $k=0,1,\dots,d-1$.
\item[(3)] If $d$ is odd and $\partial h_n=0$ for some $1 \leq n \leq d-2$ then $\sum_{k=0}^d h_k$ is even.
\item[(4)] If $\partial h_n=0$ for some $1 \leq n \leq d-2$ then
$$h_k + h_{k-1} + \cdots +h_{k-n+1} \geq \partial h_k
\ \ \mbox{ for } k=n,n+1,\dots,d-1.$$
\item[(5)]
If $\partial h_i=0$ and $h_j=0$ for some positive integers $i$ and $j$ with $i+j \leq d$ then $\sum_{k=0}^d h_k$ is even.
\item[(6)]
Suppose $\partial h_n = 0$ for some $1 \leq n < \frac d 2$.
If $(h_\ell + h_{\ell-1} + \cdots + h_{\ell-n+1}) -\partial h_\ell < n$ for some $ n \leq \ell \leq d-n$
then $\sum_{k=0}^d h_k$ is even. 
\item[(7)]
Suppose $\partial h_i=0$ and $h_j=0$ for some integers $i$ and $j$ with $0< i < \frac d 2$ and $d-i  < j <d$.
If $\partial h_\ell \leq \ell$ for some $\ell \leq d-j$ then $\sum_{k=0}^d h_k$ is even.
\end{itemize}
\end{theorem}

Note that conditions (5), (6) and (7) are unnecessary when $d$ is odd.
We also note that some special cases of the above theorem are due to Kolins \cite{Ko}.
For necessity,
(1), (2), (3) and some special cases of (4), (5) and (7) (the condition (4) for $n=1$, the condition (5) for $i=1,2$ and $j=1$,
and the condition (7) for $\ell=1$)
were proved in \cite[Sections 3, 4 and 5]{Ko}.
Also, when $\partial \hvec$ is positive or $\sum_{k=0}^d h_k$ is even,
the sufficiency of the theorem was proved in \cite[Theorem 16]{Ko}.
In particular, for even dimensional simplicial cell balls,
the sufficiency of the theorem is due to Kolins.

A simplicial poset $P$ (or a CW-complex $\Gamma(P)$) is called a \textit{simplicial cell decomposition of a topological space $X$}
if $\Gamma(P)$ is homeomorphic to $X$.
%Theorem \ref{1.1}
%characterize all possible face vectors of simplicial cell decomposition of a sphere.
It was proved in \cite{Mu} that
face vectors of simplicial cell decompositions of
some manifolds without boundary, such as product of spheres and real projective spaces,
are characterized by conditions similar to those in Theorem \ref{1.1}.
%Since Theorem \ref{1.2} characterize all possible face vectors of simplicial cell decomposition of balls,
It would be of interest to find manifolds $M$ with boundary
such that face vectors of simplicial cell decompositions of $M$
can be characterized by conditions similar to those in Theorem \ref{1.2}.

This paper is organized as follows.
In Section 2,
we prove the necessity of Theorem \ref{1.2}
by using the face ring of a simplicial poset.
In Section 3,
we prove the sufficiency of Theorem \ref{1.2}
by using constructibility and shellability.

\section{Proof of necessity}

In this section, we prove the necessity of Theorem \ref{1.2}.
To prove this, we need the face ring of a simplicial poset introduced by Stanley \cite{St}.
We assume familiarity with commutative algebra theory.
%We refer the readers to \cite{St2} for the basics on commutative algebra such as
%l.s.o.p.\ and Cohen--Macaulay property.

We recall some basic notations on commutative algebra.
Let $K$ be an infinite field,
$R=R_0 \bigoplus R_1 \bigoplus \cdots$
a finitely generated commutative graded $K$-algebra with $R_0=K$,
where $R_k$ is the graded component of $R$ of degree $k$,
and let $M$ be a finitely generated graded $R$-module.
The \textit{Hilbert series $H_M(t)$ of $M$}
is the formal power series $H_M(t)=\sum _{k = 0}^\infty (\dim_K M_k )t^k$.
The \textit{Krull dimension} of $M$, denoted $\dim M$,
is the minimal $m$ for which there exist homogeneous elements 
$\theta_1,\dots,\theta_m \in R$ of positive degrees such that $M/(\theta_1,\dots,\theta_m)M$
is a finite dimensional $K$-vector space.
When $\dim M=d$,
a sequence of homogeneous elements $\theta_1,\dots,\theta_d \in R$ of positive degrees
such that $M/(\theta_1,\dots,\theta_d)M$ is a finite dimensional $K$-vector space
is called a
\textit{homogeneous system of parameters} (\textit{h.s.o.p.\ }for short) \textit{of} $M$.
Moreover, if $\theta_1,\dots,\theta_d$ have degree $1$,
then we call $\theta_1,\dots,\theta_d$ a \textit{linear system of parameters} (\textit{l.s.o.p.\ }for short) \textit{of} $M$.
We say that $M$ is \textit{Cohen-Macaulay}
if,
for every (equivalently some) h.s.o.p.\ $\theta_1,\dots,\theta_d$ of $M$,
one has that $\theta_i$ is a non-zero divisor of $M/(\theta_1,\dots,\theta_{i-1})M$ for all $i=1,2,\dots,d$.
Also,
$R$ is said to be Cohen-Macaulay if it is a Cohen-Macaulay $R$-module.

Let $P$ be a $(d-1)$-dimensional simplicial poset
and let $A=K[x_\sigma: \sigma \in P \setminus \{ \hat 0\}]$
be the polynomial ring over an infinite field $K$ in indeterminates indexed by the elements in $P \setminus \{\hat 0\}$.
We define the grading of $A$ by $\deg x_\sigma = \rank \sigma$.
The \textit{face ring of $P$} is the quotient ring $K[P]=A/I_P$,
where $I_P$ is the ideal generated by the following elements:
\begin{itemize}
\item $x_\sigma x_\tau$, if $\sigma,\tau \in P$ have no common upper bounds in $P$.
\item $x_\sigma x_\tau - x_{\sigma \wedge \tau} \sum_\rho x_\rho$,
%for all pairs $\sigma,\tau \in P$ that have a common upper bound in $P$,
where the summation runs over the all minimal upper bounds of $\sigma$ and $\tau$ and where $\sigma \wedge \tau$ is the meet (largest lower bound) of $\sigma$ and $\tau$,
otherwise.
(We consider $x_{\sigma \wedge \tau}=1$ if $\sigma \wedge \tau=\hat 0$.)
\end{itemize}
Note that,
for all $\sigma,\tau \in P$ such that $\sigma$ and $\tau$ have a common upper bound $\rho$,
there is the unique largest lower bound of $\sigma$ and $\tau$ since $[\hat 0,\rho]$ is a Boolean algebra.
Also, since generators of $I_P$ are homogeneous,
the ring $K[P]$ is graded.
In the special case when $P$ is the face poset of a simplicial complex,
the ring $K[P]$ is isomorphic to the face ring (Stanley-Reisner ring) of a simplicial complex \cite[p.\ 53]{St2}.
See \cite[p.\ 467]{MMP}.

A simplicial poset $P$ is said to be \textit{Cohen-Macaulay}
if the ring $K[P]$ is Cohen-Macaulay.
Note that this  is equivalent to saying that the order complex of $P \setminus \{\hat 0\}$
is a Cohen-Macaulay simplicial complex.
On face rings of simplicial posets,
the following properties are known.
\begin{itemize}
\item The Krull dimension of $K[P]$ is $\dim P +1$.
\item %Let $S=K[x_\sigma : \sigma \in P,\ \rank \sigma =1]$.
%Then $K[P]$ is a finitely generated graded $S$-module of dimension $\dim P+1$.
$K[P]$ has an l.s.o.p.
\item $H_{K[P]}(t)=(h_0+h_1t+ \cdots + h_dt^d) /(1-t)^d$,
where $d=\dim P +1$ and where $(h_0,h_1,\dots,h_d)=h(P)$.
\item If $\Gamma(P)$ is homeomorphic to a ball or a sphere then $P$ is Cohen-Macaulay.
\end{itemize}
See \cite[pp.\ 325--326]{St} for the first three properties.
The last property follows from a topological criterion of Cohen-Macaulay simplicial complexes
\cite[II, Corollary 4.2 and Proposition 4.3]{St2}.

A subset $I$ of a simplicial poset $P$ is said to be an \textit{order ideal of }$P$
if $\sigma \in I$ and $\tau \leq \sigma$ imply $\tau \in I$.
Thus an order ideal of a simplicial poset is again a simplicial poset.
For a simplicial poset $P$ and elements $\sigma_1,\dots,\sigma_k \in P$,
we write $\langle\sigma_1,\dots,\sigma_k \rangle$ for the order ideal of $P$ generated by $\sigma_1,\dots,\sigma_k$,
in other words,
$$\langle \sigma_1,\dots,\sigma_k\rangle=\{\tau \in P: \mbox{ there is $1 \leq j \leq k$ such that } \tau \leq \sigma_j\}.$$
For elements $f_1,\dots,f_m$ in a ring $A$,
we write $(f_1,\dots,f_m)$ for the ideal of $A$
generated by $f_1,\dots,f_m$.
We often use the following obvious fact.

\begin{lemma}
\label{subposet}
Let $P$ be a simplicial poset  and $A=K[x_\sigma: \sigma \in P \setminus \{\hat 0\}]$.
If $Q$ is an order ideal of $P$ then
$A/(I_P+(x_\sigma: \sigma \not \in Q))$ is isomorphic to $K[Q]$ as a ring.
\end{lemma}

%We say that a simplicial poset $P$
%is a \textit{simplicial cell $d$-ball (respectively $d$-sphere)}
%if $\Gamma(P)$ is homeomorphic to a $d$-ball (respectively $d$-sphere).
Let $P$ be a simplicial cell $(d-1)$-ball.
Then each rank $d-1$ element of $P$ is covered by at most two elements.
The \textit{boundary $\partial P$ of $P$} is the order ideal
$$\partial P=\langle \sigma \in P: \rank \sigma=d-1,\ \sigma \mbox{ is covered by exactly one element in }P\rangle.$$
Since $\partial P$ is the face poset of the boundary cell complex of $\Gamma(P)$,
$\partial P$ is a simplicial cell $(d-2)$-sphere.
Moreover,
it is known that if $\hvec =(h_0,h_1,\dots,h_d)$ is the $h$-vector of $P$
then $\partial \hvec=(\partial h_0,\partial h_1,\dots,\partial h_{d-1})$
is the $h$-vector of $\partial P$.
See \cite[Section 3]{Ko}.
On $h$-vectors of simplicial cell balls,
the following result was proved in \cite[Theorem 5]{Ko}.

\begin{theorem}[Kolins]
\label{kolins}
Let $P$ be a simplicial cell $(d-1)$-ball and $\hvec=h(P)=(h_0,h_1,\dots,h_d)$.
Then
\begin{itemize}
\item[(1)] $h_0=1,$ $h_d=0$ and $h_k \geq 0$ for $k=1,2,\dots,d-1$.
\item[(2)] $\partial h_k \geq 0$ for $k=0,1,\dots,d-1$.
\item[(3)] If $d$ is odd and $\partial h_n=0$ for some $1 \leq n \leq d-2$ then $\sum_{k=0}^d h_k$ is even.
\end{itemize}
\end{theorem}

We sketch a proof of the above theorem.
%for the sake of expositions.
The vectors $\hvec$ and $\partial \hvec$ are non-negative
since $h$-vectors
of Cohen-Macaulay simplicial posets are non-negative \cite[Theorem 3.10]{St}.
Also, since $h_d(P)$ is equal to the reduced Euler characteristic of $P$ times $(-1)^{d-1}$,
$h_d(P)=0$ if $P$ is a simplicial cell ball.
Finally, (3) follows from Theorem \ref{1.1} as follows:
If $d$ is odd then 
$$\partial h_{\frac {d-1} 2}= (h_0+\cdots +h_{\frac {d-1} 2}) - (h_{\frac {d+1} 2}+ \cdots +h_d)$$
is equal to $\sum_{k=0}^d h_k$ mod $2$.
Since $\sum_{k=0}^{d-1} \partial h_k$ is even by Theorem \ref{1.1}(3),
by the symmetry $\partial h_i=\partial h_{d-1-i}$ of $\partial \hvec$,
$\partial h_{\frac {d-1} 2}$ must be even.

In the rest of this section,
we prove that conditions (4), (5), (6) and (7) in Theorem \ref{1.2} are necessary conditions 
of $h$-vectors of simplicial cell balls.

\subsection{Proof of (4)}\

\begin{lemma}
\label{trivial}
Let $P$ be a $(d-1)$-dimensional simplicial poset and
$h(P)=(h_0,\dots,h_d)$.
Let $K[P]=A/I_P$ be the face ring of $P$ and $\theta_1,\dots,\theta_d \in A_1$ an l.s.o.p.\ of $K[P]$.
If $K[P]$ is Cohen-Macaulay then, for all integers $n >0$ and $n \leq k \leq d$,
one has
$$\dim_K \big(A/\big(I_P+(\theta_1,\dots,\theta_{d-1},\theta_d^n)\big)\big)_k
=h_k+h_{k-1} + \cdots + h_{k-n+1}.$$
\end{lemma}

\begin{proof}
Let $R[i]=A/(I_P+(\theta_1,\dots,\theta_{i}))$ for $i=0,1,\dots,d-1$.
Since $R[0]=K[P]$ is Cohen-Macaulay,
we have the exact sequence
$$0 \longrightarrow R[i-1] \stackrel{\times \theta_i}{\longrightarrow} R[i-1] \longrightarrow R[i] \longrightarrow 0$$
for $i=1,2,\dots,d-1$.
The exact sequences
show that $H_{R[i]}(t)=(1-t)H_{R[i-1]}(t)$ for $i=1,2,\dots,d-1$.
Then, 
since $H_{K[P]}(t)=(h_0+h_1t+\cdots+h_dt^d)/(1-t)^d$,
we have $H_{R[d-1]}(t)=(h_0+h_1t+\cdots+h_dt^d)/(1-t)$.
This fact says	
$$\dim_K \big(R[d-1]\big)_k = h_k + h_{k-1} + \cdots + h_1 +h_0$$
for $k=0,1,\dots,d.$
Also, since $\theta_1,\dots,\theta_{d-1},\theta_d^n$ is an h.s.o.p.\ of $K[P]$,
the sequence
$$0 \longrightarrow R[d-1] \stackrel{\times \theta_d^n}{\longrightarrow} R[d-1] \longrightarrow A/\big(I_P+(\theta_1,\dots,\theta_{d-1},\theta_d^n)\big) \longrightarrow	 0$$
is exact.
Hence
\begin{eqnarray*}
\dim_K \big(A/\big(I_P+(\theta_1,\dots,\theta_{d-1},\theta_d^n)\big)\big)_k
&=& \dim_K \big(R[d-1] \big)_k - \dim_K \big( R[d-1] \big)_{k-n}\\
&=& h_k+h_{k-1} + \cdots + h_{k-n+1}
\end{eqnarray*}
for $k=n,n+1,\dots,d$,
as desired.
\end{proof}

\begin{theorem}
\label{necessary1}
Let $P$ be a simplicial cell $(d-1)$-ball and $\hvec =h(P)=(h_0,h_1,\dots,h_d)$.
If $\partial h_n=0$ for some $1 \leq n \leq d-2$ then
$$ h_k + h_{k-1} + \cdots +h_{k-n+1} \geq \partial h_k
\ \ \mbox{ for } k=n,n+1,\dots,d-1.$$
\end{theorem}

\begin{proof}
Let $A=K[x_\sigma: \sigma \in P \setminus \{ \hat 0\}]$
and $J=I_P + (x_\sigma: \sigma \not \in \partial P) \subset A$.
Then $A/J \cong K[\partial P]$ by Lemma \ref{subposet}.
For a general choice of linear forms $\theta_1,\dots,\theta_d \in A_1$,
we have (e.g.\ use a criterion of an l.s.o.p.\ for face rings \cite[Lemma 6]{Ko})
\begin{itemize}
\item $\theta_1,\dots,\theta_{d-1}$ is an l.s.o.p.\ of $A/J$.
\item $\theta_1,\dots,\theta_{d-1},\theta_d$ is an l.s.o.p.\ of $A/I_P$.
\end{itemize}
Let $\overline \Theta = \theta_1,\dots,\theta_{d-1}$.
Since
$A/J$ is a $(d-1)$-dimensional Cohen--Macaulay ring,
\begin{eqnarray*}
\label{B}
\dim_K \left(A/ \left( J+ (\overline \Theta) \right) \right)_k = h_k(\partial P)=\partial h_k
\end{eqnarray*}
for $k = 0,1,\dots,d-1$.
In particular,
since $\dim_K(A/(J+(\overline{\Theta})))_n=\partial h_n =0$ by the assumption,
we have $\theta_d^n \in J+ (\overline {\Theta}).$
Then,
since $J \supset I_P$,
we have $J+(\overline\Theta) \supset I_P + (\overline\Theta)+(\theta_d^n)$.
Thus
\begin{eqnarray*}
\dim_K \big( A/ \big( I_P + (\overline {\Theta})+ (\theta_d^n ) \big) \big)_ k \geq \dim_K \left( A/ \left( J+(\overline {\Theta}) \right) \right)_k=\partial h_k
\end{eqnarray*}
for all $k =0,1,\dots,d-1$.
Then the statement follows from Lemma \ref{trivial}.
\end{proof}

\subsection{Proof of (5), (6) and (7)}\

Let $P$ be a $(d-1)$-dimensional simplicial poset.
A rank $1$ element of $P$ is called a \textit{vertex of $P$}
and a maximal element (w.r.t.\ a partial order) of $P$ is called a \textit{facet of $P$}.
The \textit{cone of $P$} is the product of the posets 
$$C(P)=P \times \{1,2\}=\{(\sigma,i): \sigma \in P,\ i=1,2\},$$
where the ordering is defined by $(\sigma,i) \geq (\tau,j)$ if $\sigma \geq \tau$ and $i \geq j$.
Then $C(P)$ is simplicial and $\Gamma(C(P))$
is homeomorphic to the (topological) cone of $\Gamma(P)$.
Also, a straightforward computation shows $h_i(P)=h_i(C(P))$ for all $i=0,1,\dots,d$
and $h_{d+1}(C(P))=0$.

Let $P$ be a simplicial cell ball.
We define the simplicial poset $SP$ by
$$SP= P \cup C(\partial P)/\! \sim$$
where $\sim$ is the equivalence relation defined by $\sigma \sim (\sigma,1)$ for all $\sigma \in \partial P$.
In other words,
$SP$ is a simplicial cell sphere such that $\Gamma(SP)$
is obtained from $\Gamma(P)$ by coning off its boundary.

We need
the next result due to Masuda \cite[Section 5]{Ma} (see also \cite{MR} for a simplified proof).

\begin{lemma}[Masuda]
\label{2-2}
Let $P$ be a simplicial cell $(d-1)$-sphere,
$K[P]=A/I_P$ the face ring of $P$
and $\Theta=\theta_1,\dots,\theta_d \in A_1$ an l.s.o.p.\ of $K[P]$.
If $x_{v_1} x_{v_2} \cdots x_{v_d}$ is zero in $A/(I_P+(\Theta))$
for any sequence of $d$ distinct vertices $v_1,v_2,\dots,v_d$ of $P$ then the number of the facets of $P$ is even.
\end{lemma}

In the rest of this section,
we fix the following notation:
Let $P$ be a simplicial cell $(d-1)$-ball and
$$\hvec =h(P)=(h_0,h_1,\dots,h_d).$$
Let $Q=C(\partial P)$ and $SP=P \cup C(\partial P)/\! \sim$ be simplicial posets defined as above.
Also, let $A=K[x_\sigma: \sigma \in SP \setminus\{\hat 0\}]$,
$$J=I_{SP}+(x_\sigma: \sigma \in SP \setminus P)=I_{SP}+(x_\sigma: \sigma \in Q \setminus \partial P),$$
$$L= I_{SP}+(x_\sigma: \sigma \in SP \setminus Q)
=I_{SP}+(x_\sigma: \sigma \in P \setminus \partial P),$$
and let $\Theta=\theta_1,\theta_2,\dots,\theta_d \in A_1$ be a sequence of general linear forms.
Then
we have the following properties.
\begin{itemize}
\item
$A/J \cong K[P]$,
$A/L\cong K[Q]$ and
$$A/(L+J)= A/\big(I_{SP} + (x_\sigma: \sigma \in SP \setminus \partial P)\big) \cong K[\partial P].$$
\item
$\overline \Theta =\theta_1,\dots,\theta_{d-1}$ is an l.s.o.p.\ of $A/(L+J)$
and $\Theta=\theta_1,\dots,\theta_{d}$ is a common l.s.o.p.\ of $A/I_{SP}$, $A/J$ and $A/L$.
\end{itemize}

Recall that $\sum_{k=0}^d h_k$ is equal to the number of the facets of $P$.
The next lemma shows that,
to prove (5)--(7),
it is enough to study the number of the facets of $SP$.

\begin{lemma}
\label{facets}
Suppose $\partial h_i=0$ for some $1 \leq i \leq d-2$.
If the number of the facets of $SP$ is even then
$\sum_{k=0}^d h_k$ is even.
\end{lemma}

\begin{proof}
By the construction of $SP$,
the number of the facets of $P$ is equal to that of $SP$ minus that of $\partial P$.
Then the statement follows 
%since the number of the facets of $\partial P$ is even by 
from Theorem \ref{1.1}(3).
\end{proof}

\begin{lemma}
\label{int1}
Let $v_1,v_2,\dots,v_d$ be distinct vertices of $SP$.
Then
\begin{itemize}
\item[(i)] If $\partial h_i=0$ for some $1 \leq i \leq d-2$ then
$x_{v_1}x_{v_2} \cdots x_{v_i} \in L+(\Theta)$.
\item[(ii)] If $h_j=0$ for some $1 \leq j \leq d$ then
$x_{v_d} x_{v_{d-1}} \cdots x_{v_{d-j+1}} \in
J+(\Theta)$.
\end{itemize}
\end{lemma}

\begin{proof}
Recall $A/J \cong K[P]$
and $A/L \cong K[Q]$.
Since $Q$ is the cone of $\partial P$,
\begin{eqnarray*}
\label{XX1}
\dim_K \big( A/ \big( L+(\Theta) \big) \big)_k=h_k(Q)=h_k(\partial P)=\partial h_k
\end{eqnarray*}
for $k =0,1,\dots,d-1$.
Thus,  $\partial h_i=0$ implies
%\begin{eqnarray*}
%\label{2.3}
$x_{v_{1}}x_{v_2} \cdots x_{v_i} \in L+(\Theta).$
%&=&I_{SP}+ (\Theta) + (x_\sigma: \sigma \in SP \setminus Q)
%\nonumber &=&I_{SP}+ (\Theta) + (x_\sigma: \sigma \in P \setminus \partial P).
%\end{eqnarray*}
Similarly, since $\dim_K (A/(J+(\Theta)))_j=h_j$, $h_j=0$ implies
%\begin{eqnarray*}
%\label{2.4}
$x_{v_d} x_{v_{d-1}} \cdots x_{v_{d-j+1}} \in J+(\Theta)$.
%I_{SP}+ (\Theta) + (x_\tau : \tau \in Q \setminus \partial P).
%\end{eqnarray*}
%Since $i+j \leq d$, (\ref{2.3}) and (\ref{2.4}) show that
%$$
%x_{v_1} x_{v_2} \cdots x_{v_d} \in I_{SP} +(\Theta)+(x_\sigma x_\tau :
%\sigma \in P \setminus \partial P,\
%\tau \in Q \setminus \partial P).$$
%However, for all $\sigma \in P \setminus \partial P$ and $\tau \in Q \setminus \partial P$,
%we have $x_\sigma x_\tau \in I_{SP}$ since $\sigma$ and $\tau$ have no common upper bounds.
%Thus the monomial $x_{v_1} x_{v_2} \cdots x_{v_d}$ is zero in $A/(I_{SP}+(\Theta))$.
\end{proof}

Now we prove (5) and (6).

\begin{theorem}
\label{2-3}
If $\partial h_i=0$ and $h_j=0$ for some positive integers $i$ and $j$ with $i+j \leq d$ then $\sum_{k=0}^d h_k$ is even.
\end{theorem}

\begin{proof}
Let $v_1,v_2,\dots,v_d$ be distinct vertices of $SP$.
By Lemmas \ref{2-2} and \ref{facets}, it is enough to prove that 
the monomial $x_{v_1}x_{v_2}\cdots x_{v_d}$ is zero in $A/(I_{SP} +(\Theta))$.

Since $i+j \leq d$,
Lemma \ref{int1} shows
$$
x_{v_1} x_{v_2} \cdots x_{v_d} \in
LJ+(\Theta) \subset
 I_{SP} +(\Theta)+(x_\sigma x_\tau :
\sigma \in P \setminus \partial P,\
\tau \in Q \setminus \partial P).$$
However, for all $\sigma \in P \setminus \partial P$ and $\tau \in Q \setminus \partial P$,
we have $x_\sigma x_\tau \in I_{SP}$ since $\sigma$ and $\tau$ have no common upper bounds.
Thus $x_{v_1} x_{v_2} \cdots x_{v_d} \in I_{SP}+(\Theta)$.
\end{proof}

\begin{theorem}
\label{necessary7}
Suppose $\partial h_n = 0$ for some $1 \leq n < \frac d 2$.
If $(h_\ell + h_{\ell-1} + \cdots + h_{\ell-n+1}) -\partial h_\ell < n$ for some $ n \leq \ell \leq d-n$
then $\sum_{k=0}^d h_k$ is even. 
\end{theorem}

\begin{proof}
Let $v_1,v_2,\dots,v_d$ be distinct vertices of $SP$.
Then $x_{v_1} x_{v_2} \cdots x_{v_n} \in L+(\Theta)$ by Lemma \ref{int1}(i).
Since $n+ \ell \leq d$,
by the same argument as in the proof of Theorem \ref{2-3},
it is enough to prove that
$$x_{v_d}x_{v_{d-1}}\cdots x_{v_{d-\ell+1}} \in J+(\Theta).$$

Consider the exact sequence
$$
0 \longrightarrow
N\longrightarrow
A/\big(J+(\overline \Theta,\theta_d^n)\big)
\longrightarrow
A/\big(L+J+(\overline \Theta,\theta_d^n)\big)
\longrightarrow 
0,
$$
where $N=(L+J+(\overline \Theta,\theta_d^n))
/(J+(\overline \Theta,\theta_d^n))$.
Since $A/(L+J)\cong K[\partial P]$ and since $\overline \Theta$ is an l.s.o.p.\ of $A/(L+J)$,
we have $\dim_K (A/(L+J+(\overline \Theta)))_n=\partial h_n=0$ and
$$A_n=\big(L+J+(\overline \Theta)\big)_n.$$
In particular, we have
$L+J+(\overline \Theta)=L+J+(\overline \Theta,\theta_d^n)$.
Then,
by Lemma \ref{trivial} we have
\begin{eqnarray}
\label{qqq}
\hspace{16pt} \dim_K N_\ell&=&\dim_K \big(A/\big(J+(\overline \Theta,\theta_d^n)\big)
\big)_\ell
-
\dim_K \big(A/\big(L+J+(\overline \Theta,\theta_d^n)\big)
\big)_\ell\\
\nonumber & =&
(h_\ell+ h_{\ell-1}+\cdots + h_{\ell-n+1}) - \partial h_\ell \\
\nonumber &\leq& n-1,
\end{eqnarray}
where the last inequality follows from the assumption.

Let $u_k=x_{v_d} x_{v_{d-1}} \cdots x_{v_{d-k+1}}$
for $k=1,2,\dots,\ell$.
Since $A_n=(L+J+(\overline \Theta))_n$,
any element in $A$ of degree $n$ is contained in $L+J+(\overline \Theta)$.
Since $\ell \geq n$, this fact says that the elements
$$u_\ell,u_{\ell-1} \theta_d,\dots,u_{\ell-n+1}\theta_d^{n-1}$$
are contained in $L+J+(\overline \Theta)$
since they are products of elements of degree $1$.
Then \eqref{qqq}
says that the above $n$ elements are $K$-linearly dependent in $N$.
Thus there are $\alpha_0,\alpha_1,\dots,\alpha_{n-1} \in K$
with $(\alpha_0,\alpha_1,\dots,\alpha_{n-1}) \ne (0,0,\dots,0)$
such that
$$
\alpha_0 u_\ell + \alpha_1 u_{\ell-1} \theta_d + \cdots + \alpha_{n-1} u_{\ell-n+1}\theta_d^{n-1}
\in J+(\overline \Theta,\theta_d^n).$$
Let $s = \min \{k : \alpha_k \ne 0\}$.
Then there is an $h \in A$ such that
$$\theta_d^s \alpha_s u_{\ell-s} + \theta_d^{s+1} h \in J+ (\overline \Theta).$$
Since $\Theta$ is an l.s.o.p.\ of a Cohen-Macaulay ring $A/J$, $\theta_d$ is a non-zero divisor of $A/(J+(\overline \Theta))$.
Thus
$$
\alpha_s u_{\ell-s}  + \theta_d h \in  J +(\overline \Theta).
$$
This fact implies $u_{\ell-s} \in J+(\Theta)$.
%Since $\ell-s \leq d-n$, $u_{\ell-s}$ divides $u_{d-n}$.
Since $u_{\ell-s}$ divides $u_\ell$,
we have
$$
x_{v_d}x_{v_{d-1}} \cdots x_{v_{d-\ell+1}}
= u_{\ell}
%=u_{\ell-s} (x_{d-(\ell-s)}x_{d-(\ell-s)-1} \cdots x_{v_{n+1}})
\in J+(\Theta),$$
as desired.
\end{proof}

Next, we study the condition (7).
To simplify the notation,
we write
$R=A/J \cong K[P]$ and let
$$C=(L+J)/J=
\big((x_\sigma: \sigma \in P \setminus \partial P)+J \big)/J.$$
Thus $C$ is the ideal of $R$ generated by the interior faces of $P$.

\begin{lemma}
\label{2.3.1}
$C$ is a $d$-dimensional Cohen-Macaulay $A$-module such that
$H_C(t)={(h_d+h_{d-1}t + \cdots + h_0t^d)} /{(1-t)^d}$.
\end{lemma}

\begin{proof}
%Since $R$ is a finitely generated $S$-module and $C$ is a submodule of $R$,
%it is clear that $C$ is a finitely generated $S$-module.
Consider the exact sequence
\begin{eqnarray*}
\label{E2.4}
0 \longrightarrow
C \longrightarrow
R \longrightarrow
R/C \longrightarrow
0.
\end{eqnarray*}
Observe $R \cong K[P]$ and $R/C = A/(L+J) \cong K[\partial P]$.
Then
$$
H_C(t)=H_R(t)-H_{R/C}(t)=H_{K[P]}(t)-H_{K[\partial P]}(t)
=\frac {(h_d+h_{d-1}t + \cdots + h_0t^d)} {(1-t)^d}.
$$
Also, $C$ is a $d$-dimensional Cohen-Macaulay $A$-module since
$R/C$ is a $(d-1)$-dimensional Cohen-Macaulay $A$-module
and $R$ is a $d$-dimensional Cohen-Macaulay $A$-module
(e.g., use \cite[Corollary 18.6]{E}).
\end{proof}

\begin{remark}
The ideal $C$ is the canonical module of $R$ (see \cite[Chapter 3]{BH}).
This fact can be proved in the same way as in the proof of \cite[Theorem 5.7.1]{BH},
and gives an alternative proof of Lemma \ref{2.3.1}.
\end{remark}

Since $C$ is a submodule of $R$ with $\dim C=\dim R$
and since $C$ is Cohen-Macaulay,
$\Theta$ is an l.s.o.p.\ of $C$
and we have
$$\dim_K \big(C/(\Theta C)\big)_i=h_{d-i}$$
for $i=0,1,\dots,d$ by Lemma \ref{2.3.1}.
This fact shows

\begin{lemma}
\label{E2.5}
If $h_j=0$ for some $1 \leq j \leq d-1$ then
$C_{d-j}=\Theta C_{d-j-1}.$
\end{lemma}

\begin{lemma}
\label{2.3.2}
Suppose
$h_j=0$ for some $1 \leq j \leq d-1$.
If $\theta_d f  \in (C+\overline \Theta R)_{d-j}$ then,
for any linear form $l \in R_1$,
one has $l f \in C+\overline \Theta R$.
\end{lemma}

\begin{proof}
By Lemma \ref{E2.5},
$\theta_d f \in \overline \Theta R + \theta_d C$.
Thus there is $g \in C$ such that
$\theta_d (f-g) \in \overline \Theta R$.
%Then $\theta_d (f-g)$ is zero in $R/(\overline \Theta R)$.
Since $\theta_d$ is a non-zero divisor of $R/(\overline \Theta R)$,
we have $f-g \in \overline \Theta R$
and $f \in C + \overline \Theta R$.
Hence $l f \in C+\overline \Theta R$.
\end{proof}

The next
theorem completes the proof of the necessity of Theorem \ref{1.2}.

\begin{theorem}
\label{necessity6}
Suppose $\partial h_i=0$ and $h_j=0$ for some integers $i$ and $j$ with $0< i < \frac d 2$ and $d-i  < j <d$.
If $\partial h_\ell \leq \ell$ for some $\ell \leq d-j$ then $\sum_{k=0}^d h_k$ is even.
\end{theorem}

\begin{proof}
Let $v_1,v_2,\dots,v_d$ be distinct vertices of $SP$.
%By Lemmas \ref{facets} and \ref{2-2}, it is enough to prove that 
%the monomial $x_{v_1}x_{v_2}\cdots x_{v_d}$ is zero in $A/(I_{SP} +(\Theta))$.
%By Lemma \ref{int1}(i),
%we have
%$$x_{v_1}x_{v_2} \cdots x_{v_i} \in L+(\Theta).$$
%Then,
Since $i<j$,
in the same way as in the proof of Theorem \ref{2-3},
it is enough to prove that
$$x_{v_d} x_{v_{d-1}}\cdots x_{v_{j+1}} \in J + (\Theta).$$

Let $u_k=x_{v_d} x_{v_{d-1}} \cdots x_{v_{d-k+1}}$
for $k=0,1,\dots,d-j$,
where $u_0=1$.
We regard each $u_k$ as an element in $R=A/J$.
Recall that $\ell$ is an integer such that $\ell \leq d-j$ and $\partial h_\ell \leq \ell$.
Since $\overline \Theta$ is an l.s.o.p.\ of $R/C=A/(L+J) \cong K[\partial P]$,
$$\dim_K \big(R/(C+\overline \Theta R)\big)_\ell = \partial h_\ell \leq \ell.$$
Then
the elements
$$u_\ell, u_{\ell-1} \theta_d ,u_{\ell-2} \theta_d^2,\dots,u_0 \theta_d^\ell \in R_\ell$$
are $K$-linearly dependent in $R/(C+\overline \Theta R)$.
Thus there are $\alpha_0,\alpha_1,\dots,\alpha_{\ell} \in K$ 
with $(\alpha_0,\alpha_1,\dots,\alpha_{\ell}) \ne 0$
such that
\begin{eqnarray*}
\alpha_0 u_\ell  + \alpha_1 u_{\ell-1} \theta_d + \cdots + \alpha_\ell u_0 \theta_d^\ell
\in 
C+\overline \Theta R.
\end{eqnarray*}
Let $s = \min\{ k : \alpha_k \ne 0\}$
and $h=(\sum_{k=s+1}^{\ell} \alpha_k u_{\ell-k} \theta_d^{k})/\theta_d^{s+1}$,
where $h=0$ if $s=\ell$.
Then 
$\theta_d^{s} (\alpha_{s} u_{\ell-s}+ \theta_d h ) \in (C+ \overline \Theta R)_\ell$
and
$$
\theta_d^{d-j-(\ell-s)} (\alpha_{s} u_{\ell-s}+ \theta_d h ) \in (C+ \overline \Theta R)_{d-j}.
$$
By Lemma \ref{2.3.2},
we have $(u_{d-j}/u_{\ell-s}) (\alpha_s u_{\ell -s} + \theta_d h) \in C+\overline \Theta R$.
%when we replace $\theta_d^{d-j-s}$ with $u_{d-j}/u_{s}$
%in the above element,
%the element is still contained in $C+\overline \Theta R$.
Hence
$$\alpha_s u_{d-j} \in (C+ \overline \Theta R+ \theta_d R)_{d-j} \subset \Theta R,$$
where the last inclusion follows from
Lemma \ref{E2.5}.
Since $\alpha_s \ne 0$,
$u_{d-j}$ is contained in $\Theta R$.
This fact shows
$u_{d-j}=x_{v_d} x_{v_{d-1}}\cdots x_{v_{j+1}} \in J + (\Theta)$
when we regard $u_{d-j}$ as an element of $A$.
\end{proof}

\section{Proof of  sufficiency}

In this section, we prove the sufficiency
of Theorem \ref{1.2}.

Since the proof is a bit long,
we first give a short outline of the proof.
The key tools of the proof are the two main lemmas outlined below.
\begin{itemize}
\item[(I)] In Lemma \ref{3.A},
we introduce a simple way to make a new simplicial cell ball from a given simplicial cell ball
by gluing a pair of simplexes along their boundaries.
\item[(II)] In Lemma \ref{3.B},
we construct a special simplicial cell ball whose boundary has a nice property.
\end{itemize}
Then we construct a simplicial cell ball with the desired $h$-vector as follows:
First, starting from a simplex,
we gradually make a larger simplicial cell ball by using (I) repeatedly;
Then, at a certain point,
we attach a simplicial cell ball in (II) to change the shape of the boundary;
Finally, we make a simplicial cell ball with the desired $h$-vector by using (I) again.
The first technique (I) appeared in \cite{Ko}.
Indeed, if $\partial \hvec$ is positive or $\sum_{k=0}^d h_k$ is odd,
then the first technique is enough to prove the desired statement.
However, we need (II)
when $\partial \hvec$ has a zero entry and $\sum_{k=0}^d h_k$ is even.

Before proving the main lemmas,
we first study what happens to the $h$-vector when we glue two simplicial posets.
Two finite posets $P$ and $Q$ are said to be \textit{isomorphic}
if there is a bijection $f:P \to Q$, called an \textit{isomorphism},
such that,
for all $ \sigma, \tau \in P$,
$\sigma> \tau$ if and only if $f(\sigma) > f(\tau)$.
We write $P \cong Q$ if $P$ and $Q$ are isomorphic.

Let $P$ and $Q$ be simplicial posets.
Let $I \subset P$ and $J \subset Q$ be order ideals which are isomorphic as posets.
For a given isomorphism $f : I \to J$,
we define
the equivalence relation $\sim_f$ on $P \cup Q$ by $\sigma \sim_f f(\sigma)$
for all $ \sigma \in I$,
and define
$$P \cup_f  Q = (P \cup Q) /\! \sim_f.$$
Then $P \cup_f Q$ is again a simplicial poset.
We write $P \cup _I Q$ (or $P {{}_I \cup_J} Q$) for this simplicial poset $P \cup_f Q$ when the isomorphism $f$ is clear.
%We write $P \cup _I Q=P \cup_f Q$ when the isomorphism $f$ and the order ideal $J \subset Q$ are clear.

\begin{lemma}
\label{3.1}
Let $P$ and $Q$ be $(d-1)$-dimensional simplicial posets
and let $I \subset P$ and $J \subset Q$ be order ideals with $I \cong J$.
If $\dim I=d-2$ then
$$h_k(P\cup_I Q)=h_k(P) + h_k(Q) -h_k(I) + h_{k-1}(I)$$
for $k=0,1,\dots,d$,
where $h_{-1}(I)=h_d(I)=0$.
\end{lemma}

\begin{proof}
The desired formula follows from $f_k(P \cup_I Q)=f_k(P)+f_k(Q)-f_k(I)$
by straightforward computations.
\end{proof}

For a positive integer $i$,
we write $[i]=\{1,2,\dots,i\}$.
For integers $1 \leq k \leq d$,
let
$$\Delta_d(k)=
\big\{F \subset [d]: F \not \supset [k]\big\}.$$
By defining a partial order on $\Delta_d(k)$ by inclusion,
we  regard $\Delta_d(k)$ as a simplicial poset of dimension $d-2$.
In other words,
$\Delta_d(k)$ is the (abstract) simplicial complex generated by 
$\{ [d] \setminus \{i\}: i=1,2,\dots,k\}$.
%$k$ facets of the boundary complex of a $(d-1)$-simplex.
Let $\ee_0,\ee_1,\dots,\ee_d \in \ZZ^{d+1}$ be the unit vectors of $\ZZ^{d+1}$.
Thus $\ee_i$ is the vector such that $(i+1)$-th entry of $\ee_i$ is $1$ and all other entries of $\ee_i$ are $0$.
It is easy to see that $h_i(\Delta_d(k))=1$ for $i < k$ and $h_i(\Delta_d(k))=0$ for $i \geq k$.
Then the next lemma follows from Lemma \ref{3.1}.

\begin{lemma}
\label{3.2}
With the same notation as in Lemma \ref{3.1},
if $I \cong \Delta_d(k)$ then
$$h(P \cup_I Q)=h(P) + h(Q) - \ee_0 + \ee_{k}.$$
\end{lemma}

To apply Lemmas \ref{3.1} and \ref{3.2} efficiently,
it is convenient to have a nice way to construct simplicial cell balls by gluing two simplicial cell balls.
Constructibility and shellability give such nice ways.

\begin{definition}
A $(d-1)$-dimensional simplicial poset $P$ is said to be \textit{constructible} if
\begin{itemize}
\item[(i)]
$P$ is a Boolean algebra of rank $d$, or
\item[(ii)]
there are $(d-1)$-dimensional constructible simplicial posets $Q$ and $Q'$ such that
$P=Q { {}_I \cup_{I'} } Q'$ for some order ideals $I \subset Q$ and $I' \subset Q'$ and that $I \cong I'$ is a $(d-2)$-dimensional constructible simplicial poset.
\end{itemize}
\end{definition}

\begin{definition}
A simplicial poset is said to be \textit{pure}
if all its facets have the same rank.
A $(d-1)$-dimensional pure simplicial poset $P$ is said to be \textit{shellable}
if there is an order $\sigma_1,\sigma_2,\dots,\sigma_r$ of the facets of $P$,
called a \textit{shelling of $P$},
such that, for $i=2,3,\dots,r$,
$\langle \sigma_1,\sigma_2,\dots,\sigma_{i-1} \rangle \cap \langle \sigma_i \rangle$
is isomorphic to $\Delta_d(k_i)$ for some $1 \leq k_i \leq d$.
\end{definition}

Note that shellable simplicial posets are constructible since $\Delta_d(k)$ is constructible.
In the rest of this paper,
a simplicial cell $d$-ball which is constructible (respectively shellable) will be called
a \textit{constructible $d$-ball} (respectively \textit{shellable $d$-ball}).

We need the following well-known properties of constructibility and shellability.
See e.g., \cite[Theorem 11.4 and Section 12]{Bj2}.

\begin{itemize}
\item Let $P$ be a simplicial cell sphere.
If $Q \subsetneq P$ is a shellable simplicial poset with $\dim Q=\dim P$,
then $Q$ is a simplicial cell ball.
\item Boolean algebras and $\Delta_d(k)$ with $k <d$ are shellable balls.
\item Let $P$ and $Q$ be constructible $d$-balls,
and let $I \subsetneq \partial P$ and $J \subsetneq \partial Q$ be order ideals with $I \cong J$.
If $I$ is a constructible $(d-1)$-ball then $P {_{I}\cup_J} Q$ is a constructible $d$-ball.
\end{itemize}

For a finite set $X$,
we write $\# X$ for the cardinality of $X$.
The next well-known lemma, which immediately follows from Lemma \ref{3.2}
and the fact that a Boolean algebra has the $h$-vector $(1,0,\dots,0)$,
is useful to compute $h$-vectors of shellable simplicial posets.

\begin{lemma}
\label{shellablehvector}
Let $P$ be a $(d-1)$-dimensional shellable simplicial poset with a shelling $\sigma_1,\sigma_2,\dots,\sigma_r$.
Then, for $i=1,2,\dots,d$, one has
$$h_i(P)=\# \big\{ \ell \in \{2,3,\dots,r\}: 
\langle \sigma_1,\sigma_2,\dots,\sigma_{\ell-1} \rangle \cap \langle \sigma_\ell \rangle
\cong \Delta_d(i)
\big\}.$$
\end{lemma}

Now we prove our first main lemma.
We say that a simplicial poset $P$ contains a simplicial poset $Q$ if there is an order ideal
$I \subset P$ which is isomorphic to $Q$.

\begin{lemma}
\label{3.A}
Fix integers $1 \leq i \leq m \leq d$ and $1 \leq j \leq d-i$.
Let $P$ be a constructible $(d-1)$-ball such that $\partial P$ contains $\Delta_d(m)$
and let $\ell = \min\{m,i+j\}$.
Then there is a constructible $(d-1)$-ball $Q$ such that
\begin{itemize}
\item[(i)] $h(Q)=h(P) + \ee_i + \ee_{j}$, and
\item[(ii)] $\partial Q$ contains $\Delta_d(\ell)$.
\end{itemize}
\end{lemma}

\begin{proof}
By the assumption,
there is an order ideal $L \subset \partial P$ which is isomorphic to $\Delta_d(m)$.
We write
$$L=\big\{\mathcal L(F): F \subset [d],\ F \not \supset [m]\big\},$$
where the ordering is defined by $\mathcal L(F)>\mathcal L(G)$ if $F \supset G$.
Since $i \leq m$, 
$$I=\big\{\mathcal L(F): F \subset [d],\ F \not \supset [i]\big\} \cong \Delta_d(i)$$
is an order ideal of $\partial P$.
We prepare two Boolean algebras
$$A=\big\{\A(F): F \subset [d]\big\} \mbox{ and } B=\big\{\B(F): F \subset [d] \big\},$$
where the ordering is defined by inclusion on $F$.
Consider  the order ideal
$$J=\big\{ \A(F): F \subset [d],\ F \not \supset \{i+1,\dots,i+j\}\big\}
\cong \Delta_d(j).$$
Let
$$Q=
(P \cup_I A) \cup_J B
=(P \cup A \cup B) /\! \sim,$$
where $\sim$ is the equivalence relation defined by $\mathcal L(F) \sim \A(F)$ if $F \not \supset [i]$
and $\A(F) \sim \B(F)$ if $F \not \supset \{i+1,\dots,i+j\}$.
Then $Q$ is a constructible $(d-1)$-ball and has the desired $h$-vector by Lemma \ref{3.2}.

It remains to prove that $\partial Q$ contains $\Delta_d(\ell)$.
Clearly, $\partial Q$ contains an order ideal
$$\Delta
= \big\{\B(F): F \subset [d],\ F \not \supset [i] \big\}
\bigcup \big\{\mathcal L(F): F \subset [d],\ F \not \supset \{i+1,\dots,\ell\}\big\}/\!\sim.$$
We claim that $\Delta \cong \Delta_d(\ell)$.
What we must prove is that if $F \not \supset [i]$ and $F \not \supset \{i+1,\dots,\ell\}$
then $\mathcal L(F) \sim \B(F)$.
Indeed, $F \not \supset [i]$ implies $\mathcal L (F) \sim \A(F)$.
Also, since $\ell \leq i+j$,
$F \not \supset \{i+1,\dots,\ell\}$ implies $\A(F)\sim \B(F)$.
\end{proof}

We note that 
Lemma \ref{3.A} essentially appeared in the proof of \cite[Theorem 16]{Ko}.

Next, we develop a way to construct simplicial cell balls
by applying Lemma \ref{3.A} repeatedly.
Let $\hvec =(h_0,h_1,\dots,h_d) \in \ZZ_{\geq 0}^{d+1}$
be a vector such that $\sum_{k=0}^d h_k$ is even.
We write $\hvec=\ee_{i_1}+\ee_{i_2}+ \cdots + \ee_{i_a}$,
where $a=\sum_{k=0}^d h_k$ and where $i_1 \leq i_2 \leq \cdots \leq i_a$.
Consider the following decomposition of $\hvec$:
$$\hvec = (\ee_{i_1} + \ee_{i_a}) + (\ee_{i_2} + \ee_{i_{a-1}}) + \cdots + (\ee_{i_{\frac a 2}}+\ee_{i_{\frac a 2 +1}}).$$
We define the \textit{initial number $\init(\hvec)$ of $\hvec$} and the \textit{width $\width(\hvec)$ of $\hvec$} by
$$
\init(\hvec)=i_{\frac a 2}
$$
and
$$\width(\hvec)=\min \{ i_k+i_{a-k+1}: k=1,2,\dots, \mbox{${a \over 2}$}\}.$$
The following fact is straightforward.

\begin{lemma}
\label{dist}
With the same notation as above, one has
$$\width(\hvec)=
\max\{\ell : h_0+ h_1 +\cdots +h_k \leq h_d+h_{d-1} + \cdots +h_{\ell -k}
\mbox{ for all }k\}.
$$
\end{lemma}

We need the following reformulation of Lemma \ref{3.A}.

\begin{lemma}
\label{kolinsrev}
Let $\hvec=(h_0,h_1,\dots,h_d) \in \ZZ_{\geq 0}^{d+1}$ be such that
$h_0=0$, $\partial h_i \geq 0$ for $i=0,1,\dots,d-1$ and $\sum_{k=0}^d h_k$ is even.
Let $P$ be a constructible $(d-1)$-ball such that $\partial P$ contains
$\Delta_d(\ell)$ with $\ell \geq \init(\hvec)$.
Then there is a constructible $(d-1)$-ball $Q$ such that
$h(Q)=h(P)+\hvec$ and $\partial Q$ contains $\Delta_d(\min\{\ell,\width(\hvec)\})$.
\end{lemma}

\begin{proof}
We write
$$\hvec = (\ee_{i_1} + \ee_{i_a}) + (\ee_{i_2} + \ee_{i_{a-1}}) + \cdots + (\ee_{i_{\frac a 2}}+\ee_{i_{\frac a 2 +1}}),$$
where $a=\sum_{k=0}^d h_k$ and where $i_1 \leq i_2 \leq \cdots \leq i_a$.
For $k=1,2,\dots,\frac a 2$, let
$$w_k=\min \{\ell,i_1+i_a,i_2+i_{a-1},\dots,i_k + i_{a-k+1} \}$$
and let $w_0=\ell$.
The assumption $\partial h_i \geq 0$ for $i=0,1,\dots,d-1$ implies that
\begin{eqnarray}
\label{XYZ}
i_l+i_{a-l+1} \leq d
\end{eqnarray}
for all $l=1,2,\dots, \frac a 2$.

We inductively prove that,
for $k=0,1,\dots,\frac a 2$,
there is a constructible $(d-1)$-ball $Q_k$
which satisfies the following two conditions:
\begin{itemize}
\item[(i)] $h(Q_k)=h(P)+(\ee_{i_1} + \ee_{i_a}) + (\ee_{i_2} + \ee_{i_{a-1}}) + \cdots + (\ee_{i_k}+\ee_{i_{a-k+1}})$;
\item[(ii)] $\partial Q_k$ contains $\Delta_d(w_k)$.
\end{itemize}
Since $w_{\frac a 2}=\min\{\ell,\width(\hvec)\}$,
if the above statement holds then $Q_{\frac a 2}$ is the simplicial poset with the desired properties.

For $k=0$,
$Q_0=P$ satisfies the desired conditions.
Suppose that such $Q_k$, where $0\leq k < \frac a 2$, exists.
We prove the existence of $Q_{k+1}$.
Observe
$$
i_{k+1} \leq \min\{\ell, i_{a-k}\} \leq w_k,$$
where the last inequality follows from $i_{a-k}\leq i_{a-k+1} \leq \cdots \leq i_a$.
Recall that $i_{a-k} \leq d-i_{k+1}$ by \eqref{XYZ} and that $\partial Q_k$ contains $\Delta_d(w_k)$ by the induction hypothesis.
By Lemma \ref{3.A}
(apply the case when $P=Q_k$, $i=i_{k+1}$, $m=w_k$ and $j=i_{a-k}$)
there is a constructible $(d-1)$-ball $Q_{k+1}$ such that
$h(Q_{k+1})=h(Q_k)+ \ee_{i_{k+1}} + \ee_{i_{a-k}}$ and $\partial Q_{k+1}$
contains $\Delta_d(\min\{w_k,i_{k+1}+i_{a-k}\})=\Delta_d(w_{k+1})$ as desired.
\end{proof}

By considering the special case when $P$ is a Boolean algebra,
we obtain the next corollary.

\begin{corollary}[Kolins]
\label{kolins2}
Let $\hvec=(h_0,h_1,\dots,h_d) \in \ZZ_{\geq 0}^{d+1}$ be such that
$h_0=0$, $\partial h_i \geq 0$ for $i=0,1,\dots,d-1$ and $\sum_{k=0}^d h_k$ is even.
Then there is a constructible $(d-1)$-ball $P$ such that
$h(P)=\ee_0+\hvec$ and $\partial P$ contains $\Delta_d(\width(\hvec))$.
\end{corollary}

The next lemma
is our second main lemma.

\begin{lemma}
\label{3.B}
Let $n,m$ and $d$ be positive integers such that $n \leq \frac d 2$ and $d-n \leq m <d$.
Let $\hvec=(h_0,h_1,\dots,h_d) \in \ZZ^{d+1}$ be a vector such that $\sum_{k=0}^d h_k=d$,
$h_i=1$ for $0 \leq i <d-n$,
$h_i>0$ for $d-n \leq i \leq m$
and $h_i=0$ for $i>m$.
There is a constructible $(d-1)$-ball $P$ such that
\begin{itemize}
\item[(i)] $h(P)=\hvec$, and
\item[(ii)] $\partial P$ contains two ideals $I_1$ and $I_2$
that have no common facets such that $I_1 \cong \Delta_d(n)$ and $I_2 \cong \Delta_d(d-n)$.
\end{itemize}
\end{lemma}

\begin{proof}
Consider the simplicial posets
$$A = \big\{\A(F): F \subset [d+1],\ F \not \supset [n] \big\} \cong \Delta_{d+1}(n)$$
and
$$B = \big\{\B(F): F \subset [d+1],\ F \not \supset \{n+1,\dots,d\} \big\} \cong \Delta_{d+1}(d-n),$$
where the ordering is defined by inclusion on $F$.
Consider the simplicial complex
\begin{eqnarray*}
\Sigma\!\! &=&\!\! \big\langle [d+1]\setminus \{i,j\}: i\in [n],\ j\in \{n+1,\dots,d\} \big\rangle\\
&=&\!\! \big\{ F\! \subset\! [d+1]:\! \mbox{ $F \subset [d+1]\setminus \{i,j\}$ for some } i\! \in\! [n] \mbox{ and } j\! \in\! \{n\!+\!1,\dots,d\} \big\}.
\end{eqnarray*}
Let
$$I_1=\big\langle \A\big([d+1]\setminus \{i,d+1\}\big): i=1,2,\dots,n\big\rangle \cong \Delta_d(n)$$
and
$$I_2=\big\langle \B\big([d+1]\setminus \{i,d+1\}\big): i=n+1,n+2,\dots,d \big\rangle \cong \Delta_d(d-n).$$
Then
%since
%$$\partial \A = \big\langle \A([d+1]\setminus\{i,j\}): i=1,2,\dots,n,\ j=n+1,\dots,d+1 \big\rangle,$$
%we have
\begin{eqnarray*}
\partial A
&=&
\big\langle \A\big([d+1]\setminus\{i,j\}\big): i=1,2,\dots,n,\ j=n+1,\dots,d+1\big\rangle\\
&=&\big\{\A(F): F \in \Sigma\big\} \cup I_1,
\end{eqnarray*}
and similarly
$$\partial B=\big\{\B(F): F \in \Sigma\big\} \cup I_2.$$

Let ${\boldsymbol g}=(g_0,g_1,\dots,g_d)= \hvec-\sum_{i=0}^{n-1} \ee_i - \sum_{i=0}^{d-n-1} \ee_i$.
Observe that $A$ and $B$ are constructible $(d-1)$-balls with $h(A)=\sum_{i=0}^{n-1} \ee_i$
and $h(B)=\sum_{i=0}^{d-n-1} \ee_i$.
By Lemma \ref{3.1},
if there is a constructible $(d-2)$-ball $\Omega \subset \Sigma$ such that
$-h_i(\Omega)+h_{i-1}(\Omega)=g_i$ for all $i$,
then the simplicial poset
$$P=A \cup_\Omega B = (A \cup B) /\! \sim,$$
where $\sim$ is the equivalence relation defined by $\A(F) \sim \B(F)$ for $F \in \Omega$,
satisfies the desired conditions (i) and (ii).
Thus the next lemma completes the proof.
\end{proof}

\begin{lemma}
\label{3.5}
With the same notation as in the proof of Lemma \ref{3.B},
there is a shellable $(d-2)$-ball $\Omega \subset \Sigma$ such that
$-h_i(\Omega)+h_{i-1}(\Omega)=g_i$ for all $i$.
\end{lemma}

\begin{proof}
Observe
$${\boldsymbol g}=(-1,\dots,-1,0,\dots,0,h_{d-n},h_{d-n+1},\dots,h_d),$$
where $-1$ appears in the first $n$ entries.
Let
$$\alpha_\ell=h_{(d-n)+\ell-1} + h_{(d-n)+\ell} +\cdots + h_d$$
for $\ell=1,2,\dots,n+1$.
%Since $h_0+h_1+ \cdots +h_d=d$ and $h_0=\cdots =h_{d-n-1} =1$ by the assumption,
By the assumption on $\hvec =(h_0,h_1,\dots,h_d)$,
$\alpha_1=n$ and
$\alpha_\ell \leq n -(\ell-1) $ for all $\ell$.

Let
\begin{eqnarray*}
D&=&\big\{\{p,q\}: p\in [n],\ q \in \{n+1,\dots,d\},\ p+q \leq d \big\}\\
&&\bigcup \left[\bigcup_{\ell=1}^{m-(d-n)+1} \big\{ \{p,d+\ell-p\}: p=\ell,\ell+1,\dots,\ell+\alpha_{\ell}-1 \big\} \right].
\end{eqnarray*}
Since $\alpha_\ell \leq n-(\ell-1)$,
$D \subset \{\{p,q\}: p \in [n],\ q\in\{n+1,\dots,d\} \}$.
Let
$$\Omega=\big\langle [d+1] \setminus \{p,q\}: \{p,q\} \in D \big\rangle \subset \Sigma.$$
We claim that $\Omega$ satisfies the desired conditions.

We first prove that $\Omega$ is shellable.
We define the total order $\succ$ on $D$ by
$\{p,q\} \succ \{p',q'\}$, where $p<q$ and $p'<q'$,
if $p<p'$ or $p=p'$ and $q<q'$.
We show that the total order on the facets of $\Omega$
induced by $\succ$ gives a shelling of $\Omega$.
Since $h_{n-d},\dots,h_m>0$ and $h_{m+1}= \cdots =h_d=0$ by the assumption,
we have $\alpha_1> \alpha_2> \cdots > \alpha _{m-(d-n)+2}= \cdots = \alpha _{n+1}=0$.
Then, by the construction of $D$,
for any $\{p,q\} \in D$ with $p<q$,
we have $\{p-1,q\} \in D$ if $p \ne 1$ and $\{p,q-1\} \in D$ if $q \ne n+1$.
This fact shows that, for any $\{p,q\} \in D$ with $\{p,q\} \ne \{1,n+1\}$,
\begin{eqnarray*}
&&\big\langle [d+1]\setminus \{s,t\}: \{s,t\} \in D,\ \{s,t\} \succ \{p,q\} \big\rangle
\bigcap \big\langle [d+1]\setminus \{p,q\} \big\rangle\\
&&= \big\langle [d+1] \setminus \{p,q,k\}:
k=1,2,\dots,p-1,n+1,n+2,\dots,q-1 \big \rangle\\
&& \cong \Delta_{d-1}(p+q-n-2).
\end{eqnarray*}
Hence $\Omega$ is shellable.
(To see that the second line contains the first line,
use the fact that if $F$ is contained in the first line then
one has either $p'=\min([n] \setminus F) <p$ or $q'=\min(\{n+1,\dots,d\}\setminus F) <q$
since $[d+1]\setminus \{p',q'\}$ must be contained in $\{ [d+1]\setminus \{s,t\}: \{s,t\} \in D,\ \{s,t\} \succ \{p,q\} \}$.)

The above shelling and Lemma \ref{shellablehvector} show
$$h_i(\Omega)=\# \big\{\{p,q\} \in D: p+q-n-2=i \big\}$$
for $i=0,1,\dots,d-1$.
Then a simple counting shows
\begin{eqnarray*}
h_i(\Omega)=
\left\{
\begin{array}{ll}
i+1, & \mbox{ if } i \leq n-1,\\
n, & \mbox{ if } n \leq i \leq d-n-1,\\
\alpha_{i-(d-n)+2}, & \mbox{ if }  i \geq d-n.
\end{array}
\right.
\end{eqnarray*}
Note that $h_{d-n-1}(\Omega)=\alpha_1=n$.
Then we have  $-h_i(\Omega)+h_{i-1}(\Omega)=g_i$
for all $i$.

Finally,
$\Omega$ is a simplicial cell $(d-2)$-ball
since $\Omega$ can be identified with a $(d-2)$-dimensional subposet of a $(d-2)$-dimensional
simplicial cell sphere $\partial A$.
\end{proof}

\begin{remark}
The simplicial complex $\Sigma$ is the cone of the join of the boundaries of simplexes,
and
$h$-vectors of shellable subcomplexes of $\Sigma$ having the same dimension as $\Sigma$
are characterized in \cite{BFS}.
This will give an alternative proof of Lemma \ref{3.5}.
\end{remark}

\begin{theorem}
\label{sufficiency}
If $\hvec=(h_0,h_1,\dots,h_d) \in \ZZ^{d+1}$ satisfies conditions (1)--(7) in Theorem \ref{1.2},
then $\hvec$ is the $h$-vector of a simplicial cell $(d-1)$-ball.
\end{theorem}

\begin{proof}
If $\partial \hvec$ is positive or $\sum_{k=0}^d h_k$ is even,
then the assertion was proved in \cite[Theorem 16]{Ko}.
But we include a proof for completeness
(our proof is essentially the same as that of Kolins).
\medskip

{\em Case 1.}
Suppose that all entries of $\partial \hvec$ are positive.
If $\sum_{k=0}^d h_k$ is odd
then,
since $\partial(\hvec-\ee_0)$ is a non-negative vector,
Corollary \ref{kolins2} says that $\hvec =\ee_0+(\hvec -\ee_0)$
is the $\hvec$-vector of a simplicial cell $(d-1)$-ball.
Suppose that $\sum_{k=0}^d h_k$ is even.
We write
$$\hvec -\ee_0=(\ee_{i_1}+\ee_{i_a})+ (\ee_{i_2}+\ee_{i_{a-1}})+ \cdots + (\ee_{i_{\frac {a-1} 2}}+\ee_{i_{\frac {a+3} 2 }}) + \ee_{i_{\frac {a+1} 2 }},$$
where $a=\sum_{k=0}^d h_k -1$ and where $i_1 \leq \cdots \leq i_a$.
Let $\hvec'=\hvec -\ee_0 -\ee_{i_{\frac {a+1} 2 }}$.
Since $\partial(\hvec-\ee_0)$ is non-negative,
we have $i_{a-l+1} \leq d-i_l$ for $l=1,2,\dots,\frac {a-1} 2$.
Then $\partial \hvec'$ is non-negative and $\width (\hvec') \geq \min\{i_{\frac {a+3} 2 },\dots,i_a\} \geq i_{\frac {a+1} 2}$.
Thus,
by Corollary \ref{kolins2},
there is a constructible 
$(d-1)$-ball $P$ such that $h(P)=\ee_0+\hvec'= \hvec -\ee_{i_{\frac {a+1} 2}}$
and $\partial P$ contains $\Delta_d(i_{\frac {a+1} 2})$.
Let $Q$ be a simplicial cell $(d-1)$-ball obtained
from $P$
by gluing a Boolean algebra of rank $d$ along an order ideal $J \subset \partial P$ with $J \cong \Delta_d(i_{\frac {a+1} 2})$.
Lemma \ref{3.2} guarantees that the $h$-vector of $Q$
is $h(P)+\ee_{i_{\frac {a+1} 2}}=\hvec$.
\medskip

{\em Case 2.}
Suppose that $\partial \hvec$ has a zero entry and
$\sum_{k=0}^d h_k$ is even.
Let
$$n=\min\{k :\partial h_k=0\}$$
and let $\hvec '=(h_0',h_1',\dots,h_d')=\hvec -\ee_0-\ee_{d-n}$.
Note that $\hvec'$ is a non-negative vector since $\partial h_{n-1}>0$ and $\partial h_n=0$ imply $h_{d-n} > 0$.
By the same argument as in Case 1,
it is enough to prove that there is a constructible $(d-1)$-ball $P$
such that $h(P)=\hvec -\ee_{d-n}$ and $\partial P$ contains $\Delta_d(d-n)$.
By the choice of $n$,
$\partial \hvec'$ is non-negative.
Thus, by Corollary \ref{kolins2},
what we must prove is $\width (\hvec ') \geq d-n$.

By the condition (4),
for $k \geq 0$, we have
$$
h_{k+n}+ \cdots + h_{k+1}
\geq
(h_0+ h_1+ \cdots + h_{k+n}) - (h_d+ h_{d-1}+ \cdots + h_{d-k-n})$$
(we substitute $k$ by $k+n$ to the inequality in (4)).
Then we have $h_0+\cdots+h_{k} \leq h_d+ \cdots + h_{d-k-n}$ for all $k\geq 0$.
%Since $h'=h-\ee_0-\ee_{d-n}$,
This fact shows
$$
h_0'+h_1'+\cdots+h_{k}' 
%=h_0+\cdots+h_{k}
\leq
% h_d+ \cdots + h_{d-k-n}=
h_d'+ h_{d-1}'+ \cdots + h'_{d-k-n}
$$
for all $k \geq 0$.
Then Lemma \ref{dist} guarantees $\width (\hvec') \geq d-n$.
\medskip

{\em Case 3.}
Suppose that $\partial \hvec$ has a zero entry and
$\sum_{k=0}^d h_k$ is odd.
By the condition (3),
$d$ must be even.
Let
$$n=\min\{k :\partial h_k=0\}.$$
By the symmetry of $\partial \hvec$,
we have $n < \frac d 2$.
Let
$$m=\min\{k: h_k =0\}-1.$$
Thus $h_0,h_1,\dots,h_m>0$ and $h_{m+1}=0$.
Then $m \leq d-1$ since $h_d=0$
and $m \geq d-n$ by the condition (5).

By the condition (7),
$\partial h_{d-m-1} \geq d-m$.
Recall that $\partial h_n=0$, $d-m \leq n$ and $\partial h_k > 0$ for $k<n$.
Then, there is a sequence of integers
$d-m \leq s_1 <s_2 < \cdots < s_p =n$
such that
\begin{itemize}
\item $d-m > \partial h_{s_1} > \cdots > \partial h_{s_p}=0$, and
\item for any $s_{i-1} <j < s_i$ with $i \geq 2$
one has $\partial h_j \geq \partial h_{s_{i-1}}$,
and for any $d-m -1 <j < s_1$
one has $\partial h_j \geq d-m$.
\end{itemize}
(The next figure explains what $s_1,s_2,\dots,s_p$ are.)
\bigskip
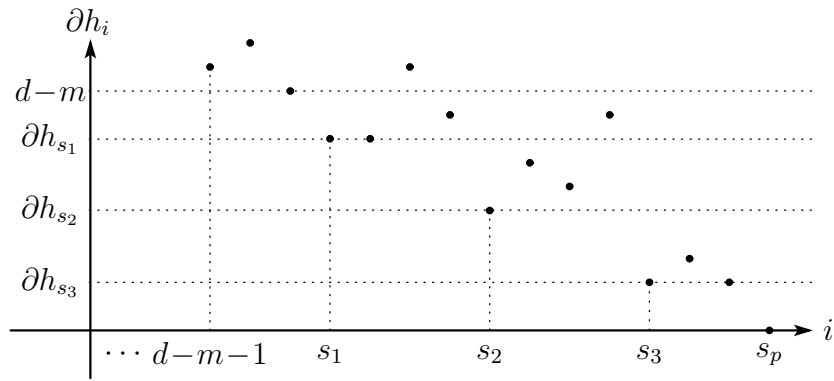
\begin{center}
\unitlength 0.1in
\begin{picture}( 50.1300, 17.5600)(  7.9800,-25.5600)
% VECTOR 1 0 3 0
% 2 1637 2556 1637 800
% 
\special{pn 13}%
\special{pa 1638 2556}%
\special{pa 1638 800}%
\special{fp}%
\special{sh 1}%
\special{pa 1638 800}%
\special{pa 1618 868}%
\special{pa 1638 854}%
\special{pa 1658 868}%
\special{pa 1638 800}%
\special{fp}%
% VECTOR 1 0 3 0
% 2 1219 2305 5400 2305
% 
\special{pn 13}%
\special{pa 1220 2306}%
\special{pa 5400 2306}%
\special{fp}%
\special{sh 1}%
\special{pa 5400 2306}%
\special{pa 5334 2286}%
\special{pa 5348 2306}%
\special{pa 5334 2326}%
\special{pa 5400 2306}%
\special{fp}%
% STR 2 0 3 0
% 3 1428 988 1428 1051 5 0
% $d\!-\!m$
\put(14.2800,-10.5100){\makebox(0,0){$d\!-\!m$}}%
% STR 2 0 3 0
% 3 1428 988 1428 1051 5 0
% $d\!-\!m$
\put(16.2800,-6.9000){\makebox(0,0){$\partial h_i$}}%
% LINE 2 2 3 0
% 2 1630 1052 5411 1052
% 
\special{pn 8}%
\special{pa 1632 1052}%
\special{pa 5412 1052}%
\special{dt 0.045}%
% STR 2 0 3 0
% 3 1428 1239 1428 1302 5 0
% $\partial s_1$
\put(14.2800,-13.0200){\makebox(0,0){$\partial h_{s_1}$}}%
% STR 2 0 3 0
% 3 1428 1615 1428 1678 5 0
% $\partial s_2$
\put(14.2800,-16.7800){\makebox(0,0){$\partial h_{s_2}$}}%
% STR 2 0 3 0
% 3 1428 1992 1428 2054 5 0
% $\partial s_3$
\put(14.2800,-20.5400){\makebox(0,0){$\partial h_{s_3}$}}%
% STR 2 0 3 0
% 3 2264 2368 2264 2430 5 0
% $d\!-\!m\!-\!1$
\put(22.6400,-24.3000){\makebox(0,0){$d\!-\!m\!-\!1$}}%
% STR 2 0 3 0
% 3 2892 2368 2892 2430 5 0
% $s_1$
\put(28.9200,-24.3000){\makebox(0,0){$s_1$}}%
% STR 2 0 3 0
% 3 3728 2368 3728 2430 5 0
% $s_2$
\put(37.2800,-24.3000){\makebox(0,0){$s_2$}}%
% STR 2 0 3 0
% 3 4564 2368 4564 2430 5 0
% $s_3$
\put(45.6400,-24.3000){\makebox(0,0){$s_3$}}%
% STR 2 0 3 0
% 3 5191 2368 5191 2430 5 0
% $s_p=n$
\put(51.9100,-24.3000){\makebox(0,0){$s_p$}}%
% STR 2 0 3 0
% 3 1820 2370 1820 2432 5 0
% $\cdots$
\put(55.0100,-23.0000){\makebox(0,0){$i$}}%
% STR 2 0 3 0
% 3 1820 2370 1820 2432 5 0
% $\cdots$
\put(18.2000,-24.3200){\makebox(0,0){$\cdots$}}%
% LINE 2 2 3 0
% 2 2264 925 2264 2305
% 
\special{pn 8}%
\special{pa 2264 926}%
\special{pa 2264 2306}%
\special{dt 0.045}%
% LINE 2 2 3 0
% 2 2892 1302 2892 2305
% 
\special{pn 8}%
\special{pa 2892 1302}%
\special{pa 2892 2306}%
\special{dt 0.045}%
% LINE 2 2 3 0
% 2 3728 1678 3728 2305
% 
\special{pn 8}%
\special{pa 3728 1678}%
\special{pa 3728 2306}%
\special{dt 0.045}%
% LINE 2 2 3 0
% 2 4564 2054 4564 2305
% 
\special{pn 8}%
\special{pa 4564 2054}%
\special{pa 4564 2306}%
\special{dt 0.045}%
% DOT 0 0 3 0
% 15 2264 925 2474 800 2683 1051 2892 1302 3101 1302 3310 925 3519 1051 3728 1678 3937 1427 4146 1552 4355 1427 4564 2054 4773 1929 4982 2054 5191 2305
% 
\special{pn 20}%
\special{sh 1}%
\special{ar 2264 926 10 10 0  6.28318530717959E+0000}%
\special{sh 1}%
\special{ar 2474 800 10 10 0  6.28318530717959E+0000}%
\special{sh 1}%
\special{ar 2684 1052 10 10 0  6.28318530717959E+0000}%
\special{sh 1}%
\special{ar 2892 1302 10 10 0  6.28318530717959E+0000}%
\special{sh 1}%
\special{ar 3102 1302 10 10 0  6.28318530717959E+0000}%
\special{sh 1}%
\special{ar 3310 926 10 10 0  6.28318530717959E+0000}%
\special{sh 1}%
\special{ar 3520 1177 10 10 0  6.28318530717959E+0000}%
\special{sh 1}%
\special{ar 3728 1678 10 10 0  6.28318530717959E+0000}%
\special{sh 1}%
\special{ar 3938 1428 10 10 0  6.28318530717959E+0000}%
\special{sh 1}%
\special{ar 4146 1552 10 10 0  6.28318530717959E+0000}%
\special{sh 1}%
\special{ar 4356 1177 10 10 0  6.28318530717959E+0000}%
\special{sh 1}%
\special{ar 4564 2054 10 10 0  6.28318530717959E+0000}%
\special{sh 1}%
\special{ar 4774 1930 10 10 0  6.28318530717959E+0000}%
\special{sh 1}%
\special{ar 4982 2054 10 10 0  6.28318530717959E+0000}%
\special{sh 1}%
\special{ar 5192 2306 10 10 0  6.28318530717959E+0000}%
% LINE 2 2 3 0
% 2 1630 1304 5411 1304
% 
\special{pn 8}%
\special{pa 1632 1304}%
\special{pa 5412 1304}%
\special{dt 0.045}%
% LINE 2 2 3 0
% 2 1630 1676 5411 1676
% 
\special{pn 8}%
\special{pa 1632 1676}%
\special{pa 5412 1676}%
\special{dt 0.045}%
% LINE 2 2 3 0
% 2 1630 2054 5411 2054
% 
\special{pn 8}%
\special{pa 1632 2054}%
\special{pa 5412 2054}%
\special{dt 0.045}%
\end{picture}%
\end{center}

Let
$${\boldsymbol \gamma}=(d-m-\partial h_{s_1}) \ee_{d-s_1} + \sum_{j=2}^p (\partial h_{s_{j-1}}- \partial h_{s_j} )\ee_{d-s_j}$$
and
$${\boldsymbol \delta}=(\delta_0,\delta_1,\dots,\delta_d) = \sum_{i=0}^m \ee_i + {\boldsymbol \gamma}.$$
By the construction of ${\boldsymbol \delta}$,
$\sum_{k=0}^d \delta_k=d+1$ and $\delta_{d-n} \geq 2$.
Let
$$\overline {\boldsymbol \delta} = {\boldsymbol \delta} -\ee_{d-n}.$$
Then $\overline {\boldsymbol \delta}$ satisfies the assumption of Lemma \ref{3.B}.
Thus there is a constructible $(d-1)$-ball $P$ such that
\begin{itemize}
\item $h(P)=\overline {\boldsymbol \delta}$.
\item $\partial P$ contains two ideals $I_1$ and $I_2$
that have no common facets such that $I_1 \cong \Delta_d(n)$ and $I_2 \cong \Delta_d(d-n)$.
\end{itemize}
We define $\hvec '=(h'_0,h_1',\dots,h'_d) \in \ZZ^{d+1}$ and $\hvec''=(h''_0,h_1'',\dots,h''_d) \in \ZZ^{d+1}$
by
$$\hvec'= (0,\dots,0,h_{n+1}-1,\dots,h_{d-n-1}-1,0,\dots,0)$$
and
$$\hvec''=  (h_0-1,\dots,h_n-1,0,\dots,0,h_{d-n}-1,\dots,h_m-1,h_{m+1},\dots,h_d)
- {\boldsymbol \gamma}.$$
We claim that the following conditions hold.
\begin{itemize}
\item[(i)] $\hvec=\hvec'+\hvec''+ \overline {\boldsymbol \delta} + \ee_{d-n}$;
\item[(ii)] $\hvec'$ and $\hvec''$ are non-negative;
\item[(iii)] $\sum_{k=0}^d h_k'$ and $\sum_{k=0}^d h''_k$ are even;
\item[(iv)] $\partial \hvec'$ and $\partial \hvec''$ are non-negative and $\partial h''_n=0$;
\item[(v)] $\width(\hvec') \geq d-n$.
\end{itemize}
We first prove that the above conditions (i)--(v) prove the desired statement.

By Corollary \ref{kolins2} and conditions (ii)--(v),
there is a constructible $(d-1)$-ball $Q$ with $h(Q)=\ee_0+\hvec'$ such that
$\partial Q$ contains an order ideal $J$ which is isomorphic to $\Delta_d(d-n)$.
Then the simplicial poset
$$R=Q {{}_J \cup_{I_2}} P$$
is a constructible $(d-1)$-ball with $h(R)=\hvec'+ \overline {\boldsymbol \delta} +\ee_{d-n}$ by Lemma \ref{3.2}.
Also, $\partial R$ contains $\Delta_d(n)$ since $I_1$ and $I_2$ have no common facets.

Since $\hvec''$ is non-negative and $\sum_{k=0}^d h_k''$ is even,
$\hvec''$ can be written in the form
$$\hvec''=(\ee_{i_1}+ \ee_{i_a}) +(\ee_{i_2}+ \ee_{i_{a-1}}) + \cdots +(\ee_{i_{\frac a 2}}+ \ee_{i_{\frac a 2+1}})$$
where $a=\sum_{k=0}^d h_k''$ and $i_1 \leq \cdots \leq i_a$.
Since $\partial h''_n=0$,
we have 
$$i_1 \leq \cdots \leq i_{\frac a 2} \leq n \mbox{ and } d-n \leq i_{\frac a 2 +1} \leq \cdots \leq i_a.$$
Thus $\init(\hvec'') \leq n$.
Since $\partial R$ contains $\Delta_d(n)$,
Lemma \ref{kolinsrev} guarantees the existence of
a constructible $(d-1)$-ball whose $h$-vector is $h(R)+\hvec''=\hvec$.
\medskip

It remains to prove (i)--(v).
Statement (i) is obvious.

{\em Proof of (ii).}
Since $h_k > 0$ for $k \leq d-n$ (see the first paragraph of Case 3),
$\hvec'$ is non-negative.
To prove that $\hvec''$ is non-negative,
it is enough to prove that $\delta_k \leq h_k$ for all $k$.
For $k > m$,
$h_k \geq 0=\delta_k$.
Also, if $k \leq m$ and $k \not \in \{d-s_1,\dots,d-s_p\}$ then $h_k \geq 1 = \delta_k$
by the choice of $m$.
Suppose $k=d-s_j$.
Since $h_{s_j} - h_{d-s_j}= \partial h_{s_j} - \partial h_{s_j-1}$,
$$h_{d-s_j} = h_{s_j} +(\partial h_{s_j-1}-\partial h_{s_j}) 
\geq 1+(\partial h_{s_{j-1}} -\partial h_{s_j}) = \delta_{d-s_j},$$
where the inequality follows from the choice of $s_1,\dots,s_p$.
(If $j=1$ then we consider that $\partial h_{s_0}=d-m$.)

{\em Proof of (iii).}
Since $\partial h_n=(h_0+\cdots+h_n) - (h_d + \cdots +h_{d-n})=0$
and since $h_0+h_1+ \cdots +h_d$ is odd,
$$h_{n+1} + \cdots + h_{d-n-1}= (h_0+ \cdots +h_d) -(h_0+ \cdots+h_n)-(h_d+ \cdots + h_{d-n})$$
is odd.
Then, since $d$ is even,
$$(h_0'+ \cdots + h_d') = (h_{n+1} + \cdots + h_{d-n-1}) -(d-2n-1)$$
is even, as desired.
Similarly, $\partial h_n=0$ implies that
$$(h_0''+ \cdots + h_d'')=
(h_0+\cdots+h_n)+(h_d+ \cdots +h_{d-n})-2n-2$$
is even.

{\em Proof of (iv).}
Observe $\partial h'_k= \partial h_k - \partial h_n$ for $k \geq n$.
Since $\partial h_n=0$, we have
\begin{eqnarray*}
\partial h_k'=
\left\{
\begin{array}{lll}
0, & \mbox{ for } k=0,1,\dots,n,\\
\partial h_k, & \mbox{ for } k=n+1,\dots,\frac d 2 -1.\\
\end{array}
\right.
\end{eqnarray*}
Hence $\partial \hvec'$ is non-negative.

For $k=0,1,\dots,d-m-1$,
$$\partial h''_k =\partial h_k -(k+1) \geq 0$$
by the condition (7).
Also, for $k=n,\dots,\frac d 2 -1$,
$$\partial h''_k=\partial h_n=0.$$
Let $k \in \{d-m,d-m+1,\dots,n-1\}$.
Then, there is an $i$ such that $s_{i-1} \leq k < s_i$, where we consider that $s_0=d-m$ and $\partial h_{s_0}=d-m$.
Then
\begin{eqnarray*}
\partial h''_k&=&
\partial h_k -(d\!-\!m) + \{ (d\!-\!m\!-\! \partial h_{s_1}) + (\partial h_{s_1}\! - \!\partial h_{s_2}) + \cdots + (\partial h_{s_{i-2}}\!-\!\partial h_{s_{i-1}})\} \\
& = & \partial h_k - \partial h_{s_{i-1}}\\
&\geq& 0,
\end{eqnarray*}
where the last inequality follows from the choice of $s_1,\dots,s_p$.
Thus $\partial \hvec''$ is non-negative.

{\em Proof of (v).}
The condition (6) says that, for $0  \leq  k \leq d-2n$,
\begin{eqnarray*}
h_{k+n}+ \cdots + h_{k+1} -n 
&\geq& \partial h_{k+n}\\
&=& \partial h_n + (h_{n+1}+ \cdots + h_{n+k}) -(h_{d-n-1}+ \cdots + h_{d-n-k})
\end{eqnarray*}
(we substitute $\ell$ by $k+n$ in the inequality in (6)).
Since $\partial h_n=0$, for $n<k  \leq d-2n$, we have
$$h_{n+1} + \cdots + h_{k} +n\leq h_{d-n-1} + \cdots + h_{d-n-k}.$$
The above inequality says
\begin{eqnarray*}
h_0'+ \cdots + h_{k}' &=&h_{n+1} + \cdots + h_{k} -(k-n)\\
& \leq& h_{d-n-1} + \cdots + h_{d-n-k} -k
 = h_d' + \cdots + h_{d-n-k}'
\end{eqnarray*}
for all $n<k<d-2n$.
Also,
if $k \leq n$ or $k \geq d-2n$ then it is clear that
$$h_0'+ \cdots + h_{k}'\leq  h_d' + \cdots + h_{d-n-k}'.$$
Then Lemma \ref{dist} guarantees $\width(\hvec')\geq d-n$.
\end{proof}
\bigskip

\noindent
\textbf{Acknowledgments}:
I would like to thank Samuel Kolins for his valuable comments on my earlier work,
which inspired me to find the complete necessary and sufficient conditions,
and for pointing out a mistake in an earlier draft.
This work was supported by KAKENHI 22740018.

\end{document}